\documentclass[reqno, 12pt]{amsart}
\pdfoutput=1
\makeatletter
\let\origsection=\section \def\section{\@ifstar{\origsection*}{\mysection}}
\def\mysection{\@startsection{section}{1}\z@{.7\linespacing\@plus\linespacing}{.5\linespacing}{\normalfont\scshape\centering\S}}
\makeatother

\usepackage{amsmath,amssymb,amsthm}
\usepackage{mathrsfs}
\usepackage{mathabx}\changenotsign
\usepackage{dsfont}

\usepackage{graphicx}
\usepackage{float}

\usepackage[dvipsnames]{xcolor}
\usepackage{hyperref}
\hypersetup{
	colorlinks,
	linkcolor={red!60!black},
	citecolor={green!60!black},
	urlcolor={blue!60!black}
}

\definecolor{codelightgray}{gray}{0.8}
\definecolor{codeverylightgray}{gray}{0.9}

\usepackage{array,multirow,colortbl,enumerate}

\usepackage{graphicx}

\usepackage[open,openlevel=2,atend]{bookmark}

\usepackage[abbrev,msc-links,backrefs]{amsrefs}
\usepackage{amsrefs}
\usepackage{doi}

\renewcommand{\PrintDOI}[1]{\doi{#1}}

\usepackage[OT2, T1]{fontenc}
\usepackage{lmodern}
\usepackage[babel]{microtype}
\usepackage[english]{babel}

\DeclareRobustCommand{\rn}[1]{  {\fontencoding{OT2}\selectfont#1}}

\usepackage{geometry}
\numberwithin{equation}{section}
\numberwithin{figure}{section}

\usepackage{enumitem}
\def\rmlabel{\upshape({\itshape \roman*\,})}

\def\alabel{\upshape({\itshape \alph*\,})}

\let\polishlcross=\l
\def\l{\ifmmode\ell\else\polishlcross\fi}

\def\paragraph#1{	\noindent\textbf{#1.}\enspace}

\let\setminus=\smallsetminus

\let\sm=\setminus

\makeatletter
\def\moverlay{\mathpalette\mov@rlay}
\def\mov@rlay#1#2{\leavevmode\vtop{   \baselineskip\z@skip \lineskiplimit-\maxdimen
		\ialign{\hfil$\m@th#1##$\hfil\cr#2\crcr}}}
\newcommand{\charfusion}[3][\mathord]{
	#1{\ifx#1\mathop\vphantom{#2}\fi
		\mathpalette\mov@rlay{#2\cr#3}
	}
	\ifx#1\mathop\expandafter\displaylimits\fi}
\makeatother

\newcommand{\dcup}{\charfusion[\mathbin]{\cup}{\cdot}}
\newcommand{\bigdcup}{\charfusion[\mathop]{\bigcup}{\cdot}}

\DeclareFontFamily{U}  {MnSymbolC}{}
\DeclareSymbolFont{MnSyC}         {U}  {MnSymbolC}{m}{n}
\DeclareFontShape{U}{MnSymbolC}{m}{n}{
	<-6>  MnSymbolC5
	<6-7>  MnSymbolC6
	<7-8>  MnSymbolC7
	<8-9>  MnSymbolC8
	<9-10> MnSymbolC9
	<10-12> MnSymbolC10
	<12->   MnSymbolC12}{}
\DeclareMathSymbol{\powerset}{\mathord}{MnSyC}{180}

\usepackage{tikz}
\usetikzlibrary{calc,decorations.pathmorphing,decorations.pathreplacing}
\usetikzlibrary{intersections}
\usetikzlibrary {arrows.meta} 
\pgfdeclarelayer{background}
\pgfdeclarelayer{foreground}
\pgfdeclarelayer{front}
\pgfsetlayers{background,main,foreground,front}

\usepackage{subcaption}

\let\epsilon=\varepsilon
\let\eps=\epsilon
\let\rho=\varrho
\let\theta=\vartheta
\let\wh=\widehat

\def\FF{{\mathds F}}

\def\ZZ{{\mathds Z}}

\def\CC{{\mathds C}}

\theoremstyle{plain}
\newtheorem{thm}{Theorem}[section]
\newtheorem{theorem}[thm]{Theorem}

\newtheorem{prop}[thm]{Proposition}
\newtheorem{proposition}[thm]{Proposition}
\newtheorem{claim}[thm]{Claim}

\newtheorem{lemma}[thm]{Lemma}

\theoremstyle{definition}
\newtheorem{rem}[thm]{Remark}
\newtheorem{dfn}[thm]{Definition}
\newtheorem{definition}[thm]{Definition}

\usepackage{accents}

\let\lra=\longrightarrow
\let\phi=\varphi

\DeclareSymbolFont{stmry}{U}{stmry}{m}{n}
\DeclareMathSymbol\arrownot\mathrel{stmry}{"58}
\DeclareMathSymbol\Arrownot\mathrel{stmry}{"59}

\def\Sym{\mathrm{Sym}}
\def\sfr{\mathrm{sf}}
\def\SFR{\mathrm{SF}}
\def\SFRR{\widetilde{\mathrm{SF}}}
\let\vn=\varnothing
\let\wh=\widehat
\def\Real{\mathrm{Re\,\,}}

\usepackage{nicefrac, xfrac}

\begin{document}
\title[Large sum-free sets in finite vector spaces I]{Large sum-free sets in finite vector spaces I.}
\author[Christian Reiher]{Christian Reiher}
\address{Fachbereich Mathematik, Universit\"at Hamburg, Hamburg, Germany}
\email{christian.reiher@uni-hamburg.de }

\author[Sofia Zotova]{Sofia Zotova}
\address{Mathematisches Institut, Universit\"at Bonn, Bonn, Germany}
\email{s87szoto@uni-bonn.de}

\begin{abstract}
	Let $p$ be a prime number with $p\equiv 2\pmod{3}$ and 
	let $n\ge 1$ be a dimension.
	It is known that a sum-free subset of $\FF_p^n$ can 
	have at most the size $\frac13(p+1)p^{n-1}$ and that, up to 
	automorphisms of $\FF_p^n$, the only extremal example is 
	the `cuboid' 
	$\bigl[\frac{p+1}3, \frac{2p-1}3\bigr]\times \FF_p^{n-1}$.
	
	For $p\ge 11$ we show that if a sum-free subset of $\FF_p^n$ 
	is not contained in such an extremal one, then its size is 
	at most $\frac13(p-2)p^{n-1}$. This bound is optimal and we 
	classify the extremal configurations. The remaining cases $p=2, 5$
	are known to behave differently. For $p=3$ the analogous question 
	was solved by Vsevolod Lev, and for $p\equiv 1\pmod{3}$ it is 
	less interesting. 
\end{abstract}
	
\maketitle
	
\section{Introduction}

A subset $A$ of an abelian group $G$ is said to be {\it sum-free} 
if the equation $x+y=z$ has no solution with $x, y, z\in A$. The 
study of this concept was begun more than a century ago by 
Schur~\cite{Schur}, who famously showed that the set of positive
integers cannot be expressed as a union of finitely many sum-free 
sets. From the 1960's onwards, a lot of activity in the subject was 
stimulated by the problems and results of Erd\H{o}s 
(see, e.g.,~\cite{E65} and the survey~\cite{TV17} 
by Tao and~Vu). Let us write $\sfr_0(G)$ for the largest possible 
size of a sum-free subset of a given finite abelian group~$G$. 
The basic question to determine this invariant of~$G$ turned out 
to be surprisingly difficult; it was solved by Green 
and Ruzsa~\cite{Gr05}.

\begin{theorem}[Green and Ruzsa]
    Let $G$ be a (nontrivial) finite abelian group. 
        \begin{enumerate}[label=\rmlabel]
       \item If $|G|$ has a prime factor $p$ with $p\equiv 2\pmod{3}$, 
       		then $\sfr_0(G)=\frac{p_0+1}{3p_0}|G|$ holds for the 
				least such prime factor $p_0$. 
       \item If the previous clause does not apply, but $|G|$ 
       		is divisible by $3$, then $\sfr_0(G)=\frac13 |G|$.
       \item If all prime factors $p$ of $|G|$ satisfy 
       		$p\equiv 1\pmod{3}$ and $m$ denotes the largest integer 
				such that $G$ possesses an element of order $m$, then 
				$\sfr_0(G)=\frac{m-1}{3m}|G|$.    
	\end{enumerate}
\end{theorem}

Building upon earlier work of Lev, \L uczak, and Schoen~\cite{LLS},
Green and Ruzsa also showed that the number of sum-free subsets 
of $G$, that is the cardinality of 
\[
	\SFR_0(G)=\{A\subseteq G\colon (A+A)\cap A=\vn\}\,,
\]
is given by 
\[
	 |\SFR_0(G)|=2^{\sfr_0(G)+o(|G|)}\,.
\]
Perhaps surprisingly, the problem to classify 
\[
	\SFRR_0(G)=\{A\in \SFR_0(G)\colon |A|=\sfr_0(G)\}\,,
\]
i.e., the sum-free subsets of $G$ of maximal size, was 
solved only less than a decade ago by Balasubramanian, 
Prakash, and Ramana~\cite{Ba16}.  
 
There are quite a few results asserting that if a sum-free subset $A$
of a specific group $G$ has size close to $\sfr_0(G)$, then it is 
contained (or almost contained) in a member of $\SFRR_0(G)$. 
Vsevolod Lev initiated a systematic quantitative 
study of this phenomenon for elementary abelian 
$p$-groups~\cites{VL05, VL23}. 
Slightly more generally, we propose to study the following hierarchy 
that starts with $\SFR_0(G)$, $\sfr_0(G)$, and $\SFRR_0(G)$. 

\begin{definition}\label{dfn:traum}
    Given a finite abelian group $G$ we define by recursion 
    on $k\ge 1$ the sets and numbers 
        \begin{itemize}
    		\item $\SFR_k(G)=\{A\in \SFR_{k-1}(G)\colon \text{ there
				is no $B\in\SFRR_{k-1}(G)$ with $A\subseteq B$}\}$,
        \item $\sfr_k(G)=\max\{|A|\colon A\in \SFR_k(G)\}$,
        \item and $\SFRR_k(G)
        	=\{A\in \SFR_k(G)\colon |A|=\sfr_k(G)\}$.
    \end{itemize}
\end{definition}

We are especially interested in the case that $k=1$ and 
$G=\FF_p^n$ is a vector space over the field $\FF_p$, where $p$ 
denotes a prime number. For $p=2$ it follows from the work of 
Davydov and Tombak~\cite{DT} that 
\[
	\sfr_1(\FF_2^n)
	=
	\begin{cases}
		0 & \text{ if } n\le 3 \cr
		5\cdot 2^{n-4} & \text{ if } n\ge 4\,.
	\end{cases}
\]
The case $p=3$ was solved by Lev~\cite{VL05}, who showed 
\[
	\sfr_1(\FF_3^n)
	=
	\begin{cases}
		0 & \text{ if } n\le 2 \cr
		5\cdot 3^{n-3} & \text{ if } n\ge 3\,.
	\end{cases}
\]

Recently Lev began working on the most difficult case, namely $p=5$.
In~\cite{VL23} he obtained 
$5^{n-1}\le \sfr_1(\FF_5^n)<\frac32\cdot 5^{n-1}$. Shortly afterwards 
Versteegen~\cite{LV23} improved these bounds 
to $28\cdot 5^{n-3}\le \sfr_1(\FF_5^n)\le 6\cdot 5^{n-1}$ 
(for $n\ge 3$), and in~\cite{RZ24b} we proved that the lower bound 
is sharp. More precisely, we have 
\[
	\sfr_1(\FF_5^n)
	=
	\begin{cases}
		0 & \text{ if } n=1 \cr
		5 & \text{ if } n=2 \cr
		28\cdot 5^{n-3} & \text{ if } n\ge 3\,.
	\end{cases}
\]

The remaining prime numbers fall into two classes depending on 
whether they are congruent to $1$ or $2$ modulo $3$. The former 
case being less interesting (as we explain in~\S\ref{sec:final}),
we focus on the latter one here. 
Despite the considerable complexity of the 
case $p=5$ it turns out that primes $p\ge 11$ with $p\equiv 2\pmod{3}$ 
admit a uniform treatment (see Theorem~\ref{thm:SF1}). 
In fact, we will also describe the corresponding extremal sets 
in $\SFRR_1(\FF_p^n)$.

Let us call two subsets of an abelian group $G$ {\it isomorphic} if there 
is an automorphism of $G$ sending one to the other. By an early observation 
of Yap~\cite{Yap}, if $p\equiv 2\pmod{3}$ and $n\ge 1$, the class $\SFRR_0(\FF_p^n)$
consists of all sets isomorphic to $[\frac{p+1}3, \frac{2p-1}3]\times \FF_p^{n-1}$.
Now we introduce a class of subsets of $\FF_p^n$ that will be shown to 
be $\SFRR_1(\FF_p^n)$, when $p\ge 11$. 
 
\begin{dfn}\label{dfn:0117}
	Let $p=6m-1\ge 11$ be a prime number and $n\ge 1$.
	A set $A\subseteq \FF_p^n$ is said to be {\it very structured} 
	if there exists a (possible empty) set $P\subseteq \FF_p^{n-1}$ 
	such that $0\not\in P+P$ and~$A$ is isomorphic to 
		\begin{align*}
		\{(2m-1, 0)\}
		\dcup 
		\{2m\}\times (\FF_p^{n-1}\sm P)
		&\dcup 
		[2m+1, 4m-3]\times\FF_p^{n-1} \\
		&\dcup 
		\{4m-2\}\times (\FF_p^{n-1}\sm\{0\})
		\dcup \{4m-1\}\times P\,.
	\end{align*}
		Furthermore, we call a set $A\subseteq \FF_p^n$ {\it structured} 
	if there are a dimension $\ell\in [n]$ and a very structured 
	set $B\subseteq \FF_p^\ell$ such that $A$ is isomorphic 
	to $B\times \FF_p^{n-\ell}$.
\end{dfn}

Notice that in the special case $n=1$ we are forced to take $P=\vn$,
so that a subset of~$\FF_p$ is structured if and only if it is 
isomorphic to $[2m-1, 4m-3]$. In general, one confirms easily 
that all structured sets $A\subseteq \FF_p^n$ have the same size 
$|A|=(2m-1)p^{n-1}$; moreover, all structured subsets of $\FF_p^n$
are sum-free. In fact, it could easily be shown that they are  
maximal sum-free sets with respect to inclusion, but we shall 
phrase our arguments in such a way that it is not necessary 
to verify this directly. Let us now state our main result. 

\begin{theorem}\label{thm:SF1}
    If $p\ge 11$ is a prime number satisfying $p\equiv 2\pmod{3}$ 
    and $n\geq 1$, then 
        \[
        \sfr_1(\FF_p^n)=\frac{p-2}{3}\cdot p^{n-1}\,.
    \]
        Furthermore, $\SFRR_1(\FF_p^n)$ is the collection of all 
    structured subsets of $\FF_p^n$. 
\end{theorem}

We would like to record that the upper bound 
$\sfr_1(\FF_p^n)<(\frac p3-\frac16+\frac1p)p^{n-1}$ was proved 
earlier by Versteegen~\cite{LV23}.
An essentially equivalent and, perhaps, more transparent version 
of our theorem reads as follows. 

\begin{theorem}\label{thm:2}
	Suppose that $p\ge 11$ is a prime number 
	satisfying $p\equiv 2\pmod{3}$ and $n\geq 1$.
	If $A\subseteq \FF_p^n$ is a sum-free set of 
	size $|A|\ge \frac13(p-2)p^{n-1}$, then either $A$ is isomorphic 
	to a subset 
	of $\bigl[\frac{p+1}3, \frac{2p-1}3\bigr]\times \FF_p^{n-1}$ 
	or $A$ is structured. 
\end{theorem}

\begin{rem}
Our Definition~\ref{dfn:traum} is inspired by an analogous 
definition of Polcyn and Ruci\'nski~\cite{Joanna} in hypergraph
Tur\'an theory. For matchings of size two their {\it higher 
order Tur\'an numbers} provide a common framework for the well-known 
theorem of Erd\H{o}s-Ko-Rado~\cite{EKR} and the subsequent 
works of Hilton-Milner~\cite{HM} as well as Han-Kohayakawa~\cite{HK}, 
which roughly correspond to our $\sfr_1(\cdot)$ and $\sfr_2(\cdot)$, 
respectively.    
\end{rem}

\subsection*{Notation}
For every positive integer $n$ we denote $\{1,\dots,n\}$ by~$[n]$.
If $A$ and $B$ are subsets of the same abelian group $G$ we set
$A+B=\{a+b\colon a\in A \text{ and } b\in B\}$. 
The symbol $\dcup$ indicates disjoint unions. 

\section{Preliminaries} \label{sec:prelim}
We start by recalling a few well-known results belonging to additive 
combinatorics and deriving some easy consequences from them. 
First, we quote a theorem due to Cauchy~\cite{Cau13}, which 
was rediscovered independently by Davenport~\cites{Dav35, Dav47}.

\begin{theorem}[Cauchy-Davenport]\label{thm:CD}
    If $p$ denotes a prime number and $A, B\subseteq\FF_p$ are nonempty, then     
        \[
    	\pushQED{\qed}
    	|A+B|\geq\min\{p,|A|+|B|-1\}\,. \qedhere
		\popQED
	 \]
	 \end{theorem}

The cases of equality were completely determined by Vosper~\cites{Vos56a, Vos56b}. 
Here we only require the main case of his characterisation.  

\begin{theorem}[Vosper]\label{thm:Vosper}
    Given a prime number $p$ let $A, B\subseteq\FF_p$ be two sets 
    such that $|A|,|B|\ge 2$ and $|A+B|\le p-2$. 
    If $|A+B|\le |A|+|B|-1$, then $A$ and $B$ are two arithmetic 
    progressions with the same common difference. \qed
\end{theorem}

Next we introduce Kneser's theorem. Given a finite nonempty 
subset~$X$ of an abelian group~$G$ we write  
\[
	\Sym(X)=\{g\in G\colon g+X=X\}
\]
for the so-called {\it symmetry group} of $X$. The statement that 
follows appears implicitly for cyclic groups in~\cite{Kn53}; in full 
generality it is stated explicitly in~\cite{Kn55}.

\begin{theorem}[Kneser]\label{Kneser}
    If $A$, $B$ are two finite nonempty subsets of an abelian 
    group~$G$ and $K=\Sym(A+B)$, then 
        \[
    			\pushQED{\qed}
        |A+B|\geq|A+K|+|B+K|-|K|\ge |A|+|B|-|K|\,. \qedhere
          \popQED
    \]
    \end{theorem}

Since $|A+K|+|B+K|$ and $|G|$ are always multiples of $|K|$, 
this has the following consequence, which could also be shown 
in a significantly easier way. 

\begin{lemma}\label{Bdumm}
	If $A$ and $B$ are two subsets of an abelian group $G$ 
	such that $|A|+|B|>|G|$, then $A+B=G$. \qed
\end{lemma}

Here is a less superficial application of Kneser's theorem. 

\begin{lemma}\label{lem:42}
	Let $A$, $B$, $C$ be three nonempty subsets of the finite vector 
	space $\FF_p^\ell$, where $p$ is prime and $\ell\ge 1$. 
	If $A+B$ is disjoint to $C$, 
	then $|A|+|B|+|C|\le (p+1)p^{\ell-1}$.
	
	Moreover, if $|A|+|B|+|C|>(p^2+1)p^{\ell-2}$, then there are a
	direct sum decomposition $\FF_p^\ell=K\oplus L$ and 
	sets $A_\star, B_\star, C_\star\subseteq L$ such that $\dim(L)=1$, 
	\begin{enumerate}[label=\rmlabel]
		\item\label{it:42i} $X\subseteq K+X_\star$ 
				for all $X\in \{A, B, C\}$, 
		\item\label{it:42ii} $|A_\star|+|B_\star|+|C_\star|=p+1$,
		\item\label{it:42iii} and $L=(C_\star-A_\star)\dcup B_\star$.
	\end{enumerate}
\end{lemma}

\begin{proof}
	We consider $K=\Sym(A+B)$. By the definition of symmetry 
	groups, $A+B$ is a union of cosets of $K$ and, therefore, disjoint 
	to $C+K$. So Kneser's theorem yields
		\begin{equation}\label{eq:1751}
		p^{\ell}\ge |A+B|+|C+K|\ge |A+K|+|B+K|+|C+K|-|K|\,.
	\end{equation}
		Furthermore, $K$ is a linear subspace of $\FF_p^\ell$ and it 
	cannot be equal to the whole space, because then $A+B=\FF_p^\ell$ 
	entailed that~$C$ had to be empty. This proves $|K|\le p^{\ell-1}$ 
	and the first assertion follows. 
	
	Now suppose, moreover, that $|A|+|B|+|C|>(p^2+1)p^{\ell-2}$.
	Due to~\eqref{eq:1751} this implies $|K|>p^{\ell-2}$ and, 
	therefore,~$K$ has the dimension $\ell-1$. So there is a 
	one-dimensional subspace $L\le \FF_p^\ell$ such that 
	$\FF_p^\ell=K\oplus L$. 
	Let $A_\star, B_\star, C_\star\subseteq L$ be the sets determined 
	by $X+K=X_\star+K$ for all $X\in \{A, B, C\}$. These sets clearly 
	satisfy~\ref{it:42i}, and~\eqref{eq:1751} provides the upper bound
	$|A_\star|+|B_\star|+|C_\star|\le p+1$. In view of 
		\[
		p^\ell<|A|+|B|+|C|\le (|A_\star|+|B_\star|+|C_\star|)p^{\ell-1}
	\]
		this actually holds with equality, which proves~\ref{it:42ii}. 
	
	Finally, since $A+B=A_\star+B_\star+K$ is disjoint to $C+K$, 
	we have $(A_\star+B_\star)\cap C_\star=\vn$, whence 
	$(C_\star-A_\star)\cap B_\star=\vn$.
	On the other, the Cauchy-Davenport theorem and~\ref{it:42ii} 
	entail 
	$|C_\star-A_\star|+|B_\star|\ge p$, and~\ref{it:42iii} follows.   
\end{proof}  

The last result of this section will only be used for the special 
prime $p=11$. 

\begin{lemma}\label{lem:ABCD}
	Suppose that $p$ is an odd prime and $\ell\ge 1$. 
	If $A$, $B$, $C$ are three nonempty subsets of $\FF_p^\ell$ 
	such that
	\begin{enumerate}[label=\alabel]
		\item\label{it:43a} $(A+B)\cap C=\vn$,
		\item\label{it:43b} $|A|+|B|+|C|>(p^2+1)p^{\ell-2}$,
		\item\label{it:43c} $|B|+p^{\ell-1}>|A-C|$,
	\end{enumerate}  
	then $B-B=\FF_p^\ell$ and $|A|+|C|\le \frac12(p+1)p^{\ell-1}$.
\end{lemma}

\begin{proof}
As this statement is invariant under automorphisms of $\FF_p^\ell$, 
Lemma~\ref{lem:42} allows us to assume that there exist 
sets $A_\star, B_\star, C_\star\subseteq \FF_p$ such that 
\begin{enumerate}[label=\rmlabel]
		\item\label{it:43i} $X\subseteq X_\star\times \FF^{\ell-1}$ 
				for all $X\in \{A, B, C\}$, 
		\item\label{it:43ii} $|A_\star|+|B_\star|+|C_\star|=p+1$,
		\item\label{it:43iii} and 
			$\FF_p=(C_\star-A_\star)\dcup B_\star$.
	\end{enumerate}  

	Let us express each of the three sets $X\in \{A, B, C\}$ in 
	the form  
		\[
		X=\bigdcup_{i\in X_\star} \{i\}\times X_i
	\]
		with appropriate subsets $X_i\subseteq \FF_p^{\ell-1}$.
	Equation~\ref{it:43ii} and Inequality~\ref{it:43b} entail 
		\begin{align*}
    &\sum_{i\in A_\star}\bigl(p^{\ell-1}-|A_i|\bigr)
    +
    \sum_{i\in B_\star}\bigl(p^{\ell-1}-|B_i|\bigr)
    +
    \sum_{i\in C_\star}\bigl(p^{\ell-1}-|C_i|\bigr) \\
    =\,&
    \bigl(|A_\star|+|B_\star|+|C_\star|\bigr)p^{\ell-1}
    -\bigl(|A|+|B|+|C|\bigr) \\
    <\,&
    (p+1)p^{\ell-1}-(p^2+1)p^{\ell-2}
    <
    p^{\ell-1}\,.
	\end{align*}
		In particular, for all $i\in A_\star$ and $j\in C_\star$
	we have 
		\[
		 \bigl(p^{\ell-1}-|A_i|\bigr)+\bigl(p^{\ell-1}-|C_j|\bigr)
		 <
		 p^{\ell-1}\,,
	\]
		i.e., $|A_i|+|C_j|>p^{\ell-1}$. In view of Lemma~\ref{Bdumm}
	this yields $A_i-C_j=\FF_p^{\ell-1}$, which in turn leads to 
		\begin{equation}\label{eq:A-C}
		A-C=(A_\star-C_\star)\times \FF_p^{\ell-1}\,.
	\end{equation}
		
	Similarly, it can be shown that 
		\begin{equation}\label{eq:B-B}
		\text{ if } |B_\star|\ge 2\,,
		\text{ then } B-B=(B_\star-B_\star)\times \FF_p^{\ell-1}\,.
	\end{equation}
		Indeed, if $d\in B_\star-B_\star$ is nonzero, then there are 
	distinct $i, j\in B_\star$ such that $d=i-j$, we can prove 
	$B_i-B_j=\FF_p^{\ell-1}$ as before, 
	and $\{d\}\times \FF_p^{\ell-1}\subseteq B-B$ follows. Moreover, 
	provided that $|B_\star|\ge 2$ there needs to exist 
	some $i\in B_\star$ with $|B_i|>\frac12 p^{\ell-1}$ 
	and $B_i-B_i=\FF_p^{\ell-1}$ shows that 
	$\{0\}\times \FF_p^{\ell-1}\subseteq B-B$ holds as well. 
	
	Now~\eqref{eq:A-C} and~\ref{it:43c} 
	yield $|B_\star|\ge |A_\star-C_\star|$ 
	and because of~\ref{it:43iii} we 
	obtain $|B_\star|\ge \frac12(p+1)\ge 2$. 
	So Lemma~\ref{Bdumm} reveals $B_\star-B_\star=\FF_p$
	and~\eqref{eq:B-B} entails $B-B=\FF_p^\ell$. 
	Moreover, we have
		\[
		\frac{|A|+|C|}{p^{\ell-1}}
		\le 
		|A_\star|+|C_\star|
		=
		p+1-|B_\star|
		\le 
		\tfrac12(p+1)\,. \qedhere
	\]
	\end{proof}

\section{First steps} \label{sec:3}

For the sake of completeness we would like to recall the 
classification of $\SFRR_0(\FF_p^n)$ when $p\equiv 2\pmod{3}$
(cf.\ e.g., \cite{Gr05}*{Lemma~5.6}).

\begin{lemma}\label{lem:5}
    Let $p$ denote a prime number with $p\equiv 2\pmod{3}$ 
    and $n\ge 1$. If $A\subseteq \FF_p^n$ with 
    $|A|\ge\frac13(p+1)p^{n-1}$ is sum-free, then $A$ is 
    isomorphic to the `cuboid' 
    $Q=\bigl[\frac{p+1}3, \frac{2p-1}3\bigr]\times \FF_p^{n-1}$. 
    
    In particular, we have $\sfr_0(\FF_p^n)=\frac13(p+1)p^{n-1}$
    and $\SFRR_0(\FF_p^n)$ is the collection of all sets isomorphic 
    to~$Q$.
\end{lemma}

\begin{proof}
	Suppose first that $n=1$. As the set $A$ is sum-free, 
	it is disjoint to $A+A$, whence    
      \[
      |A+A|\leq p-|A|
      \le
      \tfrac13(2p-1)
      \le
      2|A|-1\,.
   \]
      So, by Vosper's theorem, $A$ is an arithmetic progression
   of length $\frac{p+1}3$.
   We may assume that its common difference is $1$. 
   Together with $0\not\in A$ this yields 
   $A=\bigl[a, a+\frac{p-2}3\bigr]$ for 
   some $a\in \bigl[\frac{2p-1}3\bigr]$, and in view 
   of $2a, 2\bigl(a+\frac{p-2}3\bigr)\not\in A$ only the 
   case $a=\frac{p+1}3$ is possible. 
   
   This proves the lemma for $n=1$ and we can proceed 
   to the general case. By Lemma~\ref{lem:42} applied to $A=B=C$
   we may assume that $A$ is of the form $A_\star\times \FF_p^{n-1}$
   for some $A_\star\subseteq \FF_p$. Clearly, $A_\star$ has to 
   be a sum-free set of cardinality $\frac{p+1}3$. Thus $A_\star$ 
   is isomorphic to the interval 
   $\bigl[\frac{p+1}3, \frac{2p-1}3\bigr]$ and the result follows. 
\end{proof}

Theorem~\ref{thm:SF1} alleges that structured sets cannot be 
covered by any sets isomorphic to the cuboid $Q$. Let us briefly 
pause to check that this is indeed the case. 

\begin{lemma}
	Let $p=6m-1\ge 11$ be a prime number and $n\ge 1$. 
	If $A\subseteq \FF_p^n$ is structured, then there is 
	no $B\in \SFRR_0(\FF_p^n)$ covering $A$.
\end{lemma} 

\begin{proof}
	By Definition~\ref{dfn:0117} we can suppose that there are a
	dimension $\ell\in [n]$ and a (possibly empty) 
	set $P\subseteq \FF_p^{\ell-1}$ with $0\not\in P+P$
	such that $A=A'\times \FF_p^{n-\ell}$ holds for the 
	set
		\begin{align*}
		A'=\{(2m-1, 0)\}
		\cup 
		\{2m\} &\times (\FF_p^{\ell-1}\sm P)
		\cup 
		[2m+1, 4m-3]\times\FF_p^{\ell-1} \\
		&\cup 
		\{4m-2\}\times (\FF_p^{\ell-1}\sm\{0\})
		\cup \{4m-1\}\times P\,.
	\end{align*}
		
	If there existed some $B\in \SFRR_0(\FF_p^n)$ covering $A$,
	then Lemma~\ref{lem:5} would yield a nonzero linear form 
	$\lambda\colon\FF_p^n\lra\FF_p$ such that 
		\begin{equation}\label{eq:lambda}
		\lambda[A]\subseteq [2m, 4m-1]\,.
	\end{equation}
		
	Let $\{e_1, \dots, e_n\}$ be the standard basis of $\FF_p^n$.
	For every $i\in [2, n]$ there is a line in direction~$e_i$
	completely contained in $A$ and, therefore,~\eqref{eq:lambda}
	implies $\lambda(e_i)=0$. Thus $\mu=\lambda(e_1)$ is nonzero 
	and~\eqref{eq:lambda} leads to 
	$\mu\cdot [2m-1, 4m-3]\subseteq [2m, 4m-1]$. As the right side 
	is symmetric about the origin, 
	$\mu\cdot [2m+2, 4m]\subseteq [2m, 4m-1]$ holds as well. 
	Due to $m\ge 2$ both inclusions combine to
	$\mu\cdot [2m-1, 4m]\subseteq [2m, 4m-1]$. 
	But here the left side has more elements than the right side, 
	which is absurd.   
\end{proof}

It should be clear that in view of these two lemmata 
Theorem~\ref{thm:2} implies Theorem~\ref{thm:SF1}. 
Thus it suffices to establish Theorem~\ref{thm:2} 
in the remainder of this article. 
We commence with the one-dimensional case.
   
\begin{lemma}\label{lem:6}
    Let $p=6m-1$ be a prime number, where $m\ge 2$. 
    If $A\subseteq\FF_p$ is a sum-free set of size $|A|=2m-1$, 
    then $A$ is isomorphic either to a subset of $[2m, 4m-1]$ or 
    to $[2m-1, 4m-3]$.
\end{lemma}   

\begin{proof}	We distinguish two cases according to the cardinality 
	of $X=A\cup(-A)$.
 
   \smallskip

   {\it \hskip2em First Case: $|X|\geq 2m+2$}
        
   \smallskip

    Since the set $A$ is sum-free, it is disjoint to $A+X$ and 
    we have     
        \[
       |A+X|\le p-|A|=4m\le |A|+|X|-1\,.
    \]
	     So, by Vosper's theorem, $A$ and $X$ must be arithmetic 
    progressions with the same common difference and one checks 
    easily that $A$ needs to be isomorphic to $[2m-1, 4m-2]$.
    
    \smallskip
        
    {\it \hskip2em Second Case: $|X|\leq 2m+1$}
        
    \smallskip

    Due to $0\not\in A$ the number $|X|$ has to be even and the only 
    possibility is $|X|=2m$. It suffices to show that $X$ is 
    isomorphic to $[2m, 4m-1]$. If this failed, then 
    Lemma~\ref{lem:5} would tell us that $X$ is not sum-free. 
    
    Thus there are three elements $x_1, x_2, x_3\in X$ such that 
    $x_1+x_2=x_3$. We can suppose without loss of generality that 
    $x_3\in A$ but $x_1\not\in A$. 
    The latter means that $x_1$ is the unique member of 
    $X\sm A$ and $-x_1\in A$. 
    Now $x_2=(-x_1)+x_3\in A+A$ shows that $x_2$ is in $X\sm A$ as 
    well, whence $x_2=x_1$. Similarly, $-x_3=(-x_1)+(-x_2)\in A+A$ 
    leads to $-x_3=x_1$, so that altogether we obtain $3x_1=0$. But 
    now $0\not\in A$ and $p\ne 3$ yield a contradiction.
\end{proof}

Working towards the higher dimensional generalisation from now on, 
we proceed with a discussion of the case that the set $A$ under 
consideration lives in the `middle third' of $\FF_p^n$. This 
will be the last time in the argument where we have to work 
explicitly with the definition of structured sets. 
 
\begin{lemma}\label{prop:21}
	If $p=6m-1\ge 11$ denotes a prime number and $n\ge 1$, 
	then every set $A\subseteq [2m-1, 4m-1]\times \FF_p^{n-1}$
	belonging to $\SFR_1(\FF_p^n)$ which has at least the size  
	$(2m-1)p^{n-1}$ is structured.
\end{lemma}

\begin{proof}Let $A_{2m-1}, \dots, A_{4m-1}\subseteq \FF_p^{n-1}$ be the sets satisfying
\[
	A=\bigdcup_{i\in [2m-1, 4m-1]} \{i\}\times A_i\,.
\]
Notice that $A_{2m-1}$ cannot be empty, because 
$A\subseteq [2m, 4m-1]\times \FF_p^{n-1}$ would contradict our 
assumption that $A$ is not contained in a maximal sum-free subset 
of $\FF_p^n$ (cf.\ Lemma~\ref{lem:5}). As~$A$ is sum-free, we know 
$(A_i+A_j)\cap A_{i+j}=\vn$ whenever $i, j, i+j\in [2m-1, 4m-1]$,
whence  
\[
	|A_{2m-1}|+|A_{4m-2}|
	\le
	|A_{2m-1}+A_{2m-1}|+|A_{4m-2}|
	\le 
	p^{n-1}
\]
and, similarly,  
\[
	|A_{2m}|+|A_{4m-1}|
	\le
	|A_{2m}+A_{2m-1}|+|A_{4m-1}|
	\le 
	p^{n-1}\,.
\]
Together with the trivial upper bound $|A_i|\le p^{n-1}$ this 
yields
\begin{align*}
    (2m-1)\cdot p^{n-1}
    &\le
    |A|
    =
    \sum_{i=2m-1}^{4m-1}|A_i|\\      
    &=
    (|A_{2m-1}|+|A_{4m-2}|)+(|A_{2m}|+|A_{4m-1}|)+\sum_{i=2m+1}^{4m-3}|A_i|\\  
    &\le
    (2m-1)\cdot p^{n-1}\,.
\end{align*}           
Now all inequalities we have encountered so far are actually 
equalities, which means that
\begin{align}
    A_{2m+1}&=\dots=A_{4m-3}=\FF_p^{n-1}\,,\notag \\
    |A_{2m-1}|&=|A_{2m-1}+A_{2m-1}|\,,\label{eq:(2)}\\
    A_{4m-2}&=\FF_p^{n-1}\setminus(A_{2m-1}+A_{2m-1})\,,
    		\label{eq:(3)}\\
    |A_{2m}|&=|A_{2m-1}+A_{2m}|\,,\notag\\
    A_{4m-1}&=\FF_p^{n-1}\setminus(A_{2m-1}+A_{2m})\,.\label{eq:(5)}
\end{align}           

Since for every fixed vector $y\in\FF_p^{n-1}$ the 
map $(i, x)\longmapsto (i, x-iy)$ is an automorphism 
of~$\FF_p\times \FF_p^{n-1}$ that preserves 
$[2m-1, 4m-1]\times \FF_p^{n-1}$, we may assume $0\in A_{2m-1}$. 
Together with~\eqref{eq:(2)} this causes~$A_{2m-1}$ to be a linear 
subspace of~$\FF_p^{n-1}$. The remaining equations inform us that all 
sets~$A_i$ are unions of cosets of~$A_{2m-1}$. According to 
Definition~\ref{dfn:0117} this allows us to assume $A_{2m-1}=\{0\}$ 
and it suffices to show that~$A$ is very structured in this case.             

To this end we observe that~\eqref{eq:(3)} yields 
$A_{4m-2}=\FF_{p}^{n-1}\sm\{0\}$.
Furthermore, due to $2(4m-1)=2m-1$ (in~$\FF_p$) the set $P=A_{4m-1}$ 
satisfies $0\not\in P+P$ and, finally,~\eqref{eq:(5)} yields 
$A_{2m}=\FF_p^{n-1}\sm P$.
Altogether,~$A$ has indeed the form displayed in Definition~\ref{dfn:0117}. 
\end{proof}

\section{Further applications of Kneser's theorem}\label{sec:smp}

It turns out that Kneser's theorem allows us to relax the hypothesis 
of Lemma~\ref{prop:21} on the shape of $A$ considerably. 

\begin{prop}\label{prop:22}
	Let $p=6m-1\ge 11$ be a prime number, $n\ge 2$, 
	and $A\in\SFR_1(\FF_p^n)$.
	If $|A|\ge (2m-1)p^{n-1}$ and there is a set $I\subseteq \FF_p$
	such that $A\subseteq I\times \FF_p^{n-1}$ and $|I|\le p-3$,  
	then $A$ is structured.
\end{prop}

\begin{proof}We enumerate $\FF_p=\{b_1,\dots, b_p\}$ in such a manner that 
the cardinalities of the subsets 
$B_1, \dots, B_p\subseteq \FF_p^{n-1}$ determined by 
\[
	A=\bigdcup_{i\in [p]}\,\{b_i\}\times B_i
\]
are ordered by $|B_1|\ge |B_2|\ge \dots\ge |B_p|$.
It will be convenient to set $C_i=\{b_1,\dots,b_i\}$ for 
every $i\in[p]$ and $\ell=\max\{i\in[p]\colon B_i\ne\vn\}$.
Notice that the numbers $\beta_i=|B_i|/p^{n-2}$ fulfil  
\[
	p\ge \beta_1\ge\beta_2\ge\dots\ge \beta_\ell>0
	\quad \text{ and } \quad 
	\beta_{\ell+1}=\dots=\beta_p=0\,.
\]
Our assumptions $|A|\ge (2m-1)p^{n-1}$ and $|I|\le p-3$ entail 
\begin{equation}\label{eq:0143}
	\sum_{i=1}^\ell\beta_i\ge (2m-1)p
	\quad \text{ and }\quad
	\ell\le p-3\,,
\end{equation}
respectively. 

\begin{claim}\label{clm:1}
    Let $r,s,t\in[\ell]$. If $r$ is odd and $r+s+t\ge p+1$,
	 then $\beta_r+\beta_s+\beta_t\le p+1$ and $\beta_r+\beta_s\le p$. 
\end{claim}

\begin{proof}
    We distinguish two cases.
   
    \smallskip
 
    {\it \hskip2em First Case: $0\not\in C_r$}
        
    \smallskip

    As $r$ is odd, we have $C_r\ne -C_r$ and the 
    union $X=C_r\cup(-C_r)$ has at least the size $r+1$. 
    Thus the Cauchy-Davenport theorem yields
        \[
    		|X+C_t|\ge \min\{r+t,p\}\ge p+1-s\,,
    \]
        and, consequently, the sets $X+C_t$ and $C_s$ cannot be disjoint. 
    This means that there exist some numbers $\rho\le r$, 
    $\sigma\le s$, $\tau\le t$, and a sign $\eps\in\{-1, +1\}$ such 
    that $\eps b_\rho+b_\tau=b_\sigma$. 
    In view of $(\eps A+A)\cap A=\vn$ 
    we conclude $(\eps B_\rho+B_\tau)\cap B_\sigma=\vn$ 
    and Lemma~\ref{lem:42} reveals 
        \[
    	\beta_r+\beta_s+\beta_t
		\le 
		\frac{|B_\rho|+|B_\sigma|+|B_\tau|}{p^{n-2}}
		\le
		p+1\,.
	 \]
	 	 Moreover, we obtain  
	 	 \[
	 	\beta_r+\beta_s
		\le 
		\frac{|\eps B_\rho+B_\tau|+|B_\sigma|}{p^{n-2}}
		\le
		p\,.
	\]
	 
   \smallskip

   {\it \hskip2em Second Case: $0\in C_r$}

   \smallskip

   Take $\rho\leq r$ such that $b_{\rho}=0$. Following the arguments 
   from the previous case but using the equalities $b_{\rho}+b_s=b_s$ 
   and $b_{\rho}+b_t=b_t$ we see 
      \[
   		\beta_r+\beta_s+\beta_t
		\le
		\frac{\beta_{\rho}+2\beta_s}{2}+\frac{\beta_{\rho}+2\beta_t}{2}
		\le
		p+1
	\]
		and 
		\[
		\beta_r+\beta_s
		\le
		\frac{|B_\rho+B_s|+|B_s|}{p^{n-2}}
		\le
		p\,. \qedhere
	\]
	\end{proof}

\begin{claim}\label{clm:2}
    We have $\ell\le 2m+2$.
\end{claim}

\begin{proof}
    Assume for the sake of contradiction that $\ell\ge 2m+3$. 
    Three applications of Claim~\ref{clm:1} disclose
        \begin{align*}
    		\beta_{2m-3}+\beta_{2m}+\beta_{2m+3} &\le 6m\,, \\
			\beta_{2m-2}+\beta_{2m-1}+\beta_{2m+4}
			\le
			\beta_{2m-2}+\beta_{2m-1}+\beta_{2m+3} &\le 6m\,, \\
			\beta_{2m+1}+\beta_{2m+2}
			\le
			\tfrac23(2\beta_{2m+1}+\beta_{2m+2}) &\le 4m\,,
	 \end{align*}
	 	 and by adding these estimates we conclude
	 	 \begin{equation}\label{eq:2336}
	 		\sum_{i=2m-3}^{2m+4}\beta_i\le 16m\,.
	 \end{equation}
	 	 Next we contend 
	 	 \begin{equation}\label{eq:1}
       \beta_i+\beta_{p-2i-2}+\beta_{p-2i-1}
       \le p+1
   \end{equation}
      for every $i\in[2m-4]$. If $p-2i-2>\ell$ this follows immediately 
   from $\beta_i\le p$; on the other hand, if $p-2i-2\le\ell$, 
   then Claim~\ref{clm:1} implies the even stronger estimate 
   $\beta_i+2\beta_{p-2i-2}\le p+1$, because
      \[
   	i+2(p-2i-2)=2p-3i-4\ge 2p-3(2m-4)-4=p+7\,.
   \]
      Summing up the Inequalities~\eqref{eq:1} for all $i\in[2m-4]$
   and adding~\eqref{eq:2336} we arrive at
      \[
   	\sum_{i=1}^{p-3}\beta_i
		\le 
		(2m-4)(p+1)+16m
		=
		(2m-1)p-1
		\overset{\eqref{eq:0143}}{<}
		\sum_{i=1}^{\ell}\beta_i\,,
	\]
		which contradicts $\ell\le p-3$.
\end{proof}

\begin{claim}\label{clm:3}
We have $|B_{2m-1}|>\frac{p^{n-1}}2$.
\end{claim}
             
\begin{proof}
    If $\ell\le 2m$, we subtract $\beta_i\le p$ for all $i\in [2m-2]$ 
    from the left estimate in~\eqref{eq:0143}, thus getting 
    $\beta_{2m-1}+\beta_{2m}\ge p$.
    This implies $\beta_{2m-1}\ge \frac p2$ or, in other words, 
    $|B_{2m-1}|\ge \frac{p^{n-1}}2$. As~$p$ is odd, equality
    is impossible, and our claim follows. 
    
    Due to Claim~\ref{clm:2} it remains to discuss the 
    case $\ell\in \{2m+1, 2m+2\}$.
    Claim~\ref{clm:1} tells us $\beta_{2m-j}+\beta_{2m+j-1}\le p$ 
    for all $j\in\{1, 2, 3\}$ and thus we have 
        \[
    	(2m-1)p
		\le 
		\sum_{i=1}^{2m+2}\beta_i
		=
		\sum_{i=1}^{2m-4}\beta_i
			+\sum_{j=1}^3(\beta_{2m-j}+\beta_{2m+j-1})
		\le 
		\bigl((2m-4)+3\bigr)p
		=
		(2m-1)p\,.
	\]
		Equality needs to hold throughout this estimate and, 
	in particular, we have $\beta_{2m-1}+\beta_{2m}=p$, 
	which allows us to conclude the argument as in the 
	first case. 
\end{proof}        
        
From Claim~\ref{clm:3} and Lemma~\ref{Bdumm} 
we learn $B_i\pm B_j=\FF_p^{n-1}$ whenever $i,j \in [2m-1]$.
This shows that $C_{2m-1}\pm C_{2m-1}$ is disjoint to $C_\ell$ and,
in particular, that $C_{2m-1}$ is sum-free. Owing to 
Lemma~\ref{lem:6} we can assume that $C_{2m-1}$ is either the 
interval $[2m-1, 4m-3]$ or a subset of~$[2m, 4m-1]$. 

Suppose first that $C_{2m-1}=[2m-1,4m-3]$. Due to $m\ge 2$ this 
implies 
\[
	C_{2m-1}\pm C_{2m-1}
	=
	\FF_p\sm C_{2m-1}
\]
and thus we have $\ell=2m-1$ and 
$A\subseteq [2m-1, 4m-3]\times \FF_p^{n-1}$. 
Because of $|A|\ge (2m-1)p^{n-1}$ this needs to hold with equality. 
Since $[2m-1, 4m-3]$ is a very structured 
subset of~$\FF_p$, it follows that~$A$ is indeed structured.    

It remains to deal with the case that~$C_{2m-1}$ arises from 
$[2m,4m-1]$ by deleting one element. For reasons of symmetry we 
can suppose $2m\in C_{2m-1}$. Now a short calculation in~$\FF_p$ 
yields 
\begin{align*}
    C_{2m-1}+(\pm C_{2m-1})&=C_{2m-1}+[2m,4m-1]\\
    &\supseteq[4m,8m-3]
    =[0,2m-2]\cup[4m,p-1]\,.
\end{align*}     
Therefore, $C_\ell$ is a subset of $[2m-1, 4m-1]$, 
whence $A\subseteq [2m-1, 4m-1]\times \FF_p^{n-1}$.
By Lemma~\ref{prop:21} this implies that $A$ is structured.  
\end{proof}           

When one attempts to continue the foregoing argument, it 
turns out that even if the smallest set~$I$ 
with $A\subseteq I\times \FF_p^{n-1}$ has at least the size $p-2$
there is something interesting we can say. 

\begin{proposition}\label{prop:23}
	Given a prime number $p=6m-1\ge 11$ and a dimension $n\ge 2$ 
	let~$(A_k)_{k\in \FF_p}$
	be a family of subsets of $\FF_p^{n-1}$. 
	If $A=\bigdcup_{k\in \FF_p} \{k\}\times A_k$ 
	is a sum-free set of size $|A|\ge (2m-1)p^{n-1}$
	and at most two of the sets $A_k$ are empty, 		
	then there exists a real number~$U$ such 
	that $\sum_{k\in \FF_p}\big||A_k|-U\big|\le 2p^{n-1}$.
\end{proposition}

\begin{proof}	We start with the same enumeration $\FF_p=\{b_1,\dots, b_p\}$ 
	and decomposition 
		\[
		A=\bigdcup_{i\in [p]}\,\{b_i\}\times B_i
	\]
		as in the previous proof. 
	Setting again $\ell=\max\{i\in[p]\colon B_i\ne\vn\}$ 
	and $\beta_i=|B_i|/p^{n-2}$ for every $i\in [p]$ we have 
		\[
		p\ge \beta_1\ge\beta_2\ge\dots\ge \beta_\ell>0\,,
		\quad \quad 
		\beta_{\ell+1}=\dots=\beta_p=0\,,
	\]
		and 
		\begin{equation}\label{eq:2212}
		\sum_{i=1}^\ell\beta_i\ge (2m-1)p\,.
	\end{equation}
		The assumption that at most two of the sets $A_k$ be empty 
	yields $\ell \ge p-2$. As the proof of Claim~\ref{clm:1} 
	does not depend on the value of $\ell$, that statement is 
	still available to us.
	
	For all pairs $(i, j)\in [m]\times [3]$ 
	we have $(2i-1)+2i+(p+j-3i)=p+i+j-1\ge p+1$ 
	and, therefore, Claim~\ref{clm:1} yields 
		\[
		\beta_{2i-1}+\beta_{2i}+\beta_{p+j-3i}
		\le 
		p+1\,.
	\]
		Summing over all pairs $(i, j)$ we deduce
		\begin{equation}\label{eq:0038}
		3\sum_{k=1}^{2m}\beta_k+\sum_{k=3m}^{p}\beta_k
		\le
		3m(p+1)=18m^2\,.
	\end{equation}
		Due to $3\beta_{2m+1}\le p+1$ we have $\beta_{2m+1}\le 2m$ 
	and, consequently,
		\[
		3\sum_{k=2m+1}^{3m-1}\beta_k+\beta_{3m}\le 2(3m-2)m\,.
	\]
		By adding~\eqref{eq:0038} we infer 
		\[
		3\sum_{k=1}^{3m-1}\beta_k+2\beta_{3m}+\sum_{k=3m+1}^{p}\beta_k
		\le 
		24m^2-4m=4mp\,.
	\]
		Now we subtract the double of~\eqref{eq:2212} and multiply 
	by $p^{n-2}$, thereby obtaining 
		\[
		\sum_{k=1}^p \big||B_k|-|B_{3m}|\big|
		=
		p^{n-2}\Bigl(\sum_{k=1}^{3m-1}\beta_k-\sum_{k=3m+1}^{p}\beta_k\Bigr)
		\le 
		2p^{n-1}\,.
	\]
		This shows that the number $U=|B_{3m}|$ is as required. 
\end{proof}

Intuitively, the estimate 
$\sum_{k\in \FF_p}\big||A_k|-U\big|\le 2p^{n-1}$ means that the 
map $k\longmapsto |A_k|$ does not deviate too much from the constant 
function whose value is always $U$, which gives us a fair amount of 
control over the Fourier coefficients of $A$. 
In fact, for $p\ge 17$ this information is enough to conclude 
the proof of Theorem~\ref{thm:2}, but in the special case $p=11$ 
we need to argue much more carefully to achieve the same goal.

\begin{prop}\label{prop:24}
	Suppose that $n\ge 2$ and at most two of the sets 
	$A_0, \dots, A_{10}\subseteq\FF_{11}^{n-1}$ are empty. 
	If $A=\bigdcup_{k\in \FF_{11}}\{k\}\times A_k$ is a sum-free 
	subset of $\FF_{11}^n$ of size $|A|\ge 3\cdot 11^{n-1}$, then 
		\[
		\sum_{k=1}^{10}\left(1-\cos\frac{2k\pi}{11}\right)|A_k|
		<
		\frac{11|A|}8\,.
	\]
	\end{prop}

\begin{proof}
The quotients $\alpha_i=\frac{|A_i|}{11^{n-2}}$ 
and $\alpha=\frac{|A|}{11^{n-2}}$ satisfy $\alpha_i\le 11$
for all $i\in \FF_{11}$ and 
\[
	33\le\alpha=\sum_{i\in \FF_{11}}\alpha_i\,.
\]
	
\begin{claim}\label{nvmifempty}
   Whenever $i, j\in \FF_{11}$ are distinct and nonzero, we 
   have $\alpha_i+\alpha_j\le 11$; if additionally $i+j\ne 0$, 
	then $\alpha_i+\alpha_j+\alpha_{i+j}\le 12$.
\end{claim}

\begin{proof}
	If we had $\alpha_i+\alpha_j>11$, or in other words 
	$|A_i|+|A_j|>11^{n-1}$, then Lemma~\ref{Bdumm} would imply 
	$A_{i+j}=A_{i-j}=A_{j-i}=\vn$, contrary to the assumption 
	that at most two of the sets $A_k$ be empty. 
	The second assertion follows from the first one if one 
	of $A_i$, $A_j$, $A_{i+j}$ is empty and from Lemma~\ref{lem:42} 
	otherwise.
\end{proof}

By adding symmetric pairs of estimates derived from 
Claim~\ref{nvmifempty} we find  
\begin{align*}
	(\alpha_1+\alpha_{10})+2(\alpha_5+\alpha_6) &\le 24\,,\\
	(\alpha_2+\alpha_9)+(\alpha_4+\alpha_7)+(\alpha_5+\alpha_6)
		&\le 24\,, \\
	(\alpha_3+\alpha_8)+2(\alpha_4+\alpha_7)&\le 24\,,\\
   2(\alpha_3+\alpha_8)+(\alpha_5+\alpha_6)&\le 24\,.
\end{align*}
In terms of the sums $\beta_i=\alpha_i+\alpha_{11-i}$ 
(where $i\in [5]$) this simplifies to 
\begin{equation}\label{eq:1125}
	\max\{\beta_1+2\beta_5, \beta_2+\beta_4+\beta_5, 
		\beta_3+2\beta_4, 2\beta_3+\beta_5\}\le 24\,.
\end{equation}

Clearly, 
\begin{equation}\label{eq:1156}
	\beta_1+\beta_2+\beta_3+\beta_4+\beta_5\le\alpha
\end{equation}
and assuming that our proposition fails we have
\[
	\frac{11\alpha}8
	\le
	\sum_{k=1}^5\Bigl(1-\cos\frac{2k\pi}{11}\Bigr)\beta_k\,.
\]
Approximating the cosines we derive 
\begin{equation}\label{eq:1150}
	1.375\cdot\alpha
	\le
	0.159\cdot\beta_1+0.585\cdot\beta_2+1.143\cdot\beta_3
	+1.655\cdot\beta_4+1.96\cdot\beta_5\,.
\end{equation}

In combination with~\eqref{eq:1125} and~\eqref{eq:1156} this tells us 
\begin{align*}
    1.375\cdot\alpha
    &\le
    0.159\cdot(\beta_1+\beta_2+\beta_3+\beta_4+\beta_5)
    +0.426\cdot(\beta_2+\beta_4+\beta_5)\\
    &\qquad+
    0.535\cdot(\beta_3+2\beta_4)
    		+0.225\cdot(2\beta_3+\beta_5)+1.15\cdot\beta_5\\
    &\le
    0.159\cdot \alpha+(0.426+0.535+0.225)\cdot 24+1.15\cdot\beta_5\,,
\end{align*}
i.e., $1.216\cdot \alpha\le28.464+1.15\cdot \beta_5$. 
Since $\alpha\ge 33$, this leads to   
\begin{equation}\label{eq:beta5}
    10<\beta_5\,.
\end{equation}
Using the additional inequality $\beta_5\leq11$, which follows from 
Claim~\ref{nvmifempty}, we obtain similarly
\begin{align*}
    45.375
    &\le 
    1.375\cdot\alpha \\
    &\le
    0.159\cdot(\beta_1+2\beta_5)
    		+0.585\cdot(\beta_2+\beta_4+\beta_5) \\
    &\qquad+
    0.535\cdot(\beta_3+2\beta_4)
    		+0.304\cdot(2\beta_3+\beta_5)+0.753\cdot\beta_5\\
    &\le
    (0.159+0.535+0.304)\cdot 24+0.753\cdot 11 
    		+0.585\cdot(\beta_2+\beta_4+\beta_5)\,,
\end{align*}
whence
\begin{equation}\label{eq:224}
    \beta_2+\beta_4+\beta_5 > 22.4\,.
\end{equation}

Since $A_4-A_9$ and $A_6$ are disjoint, we have 
\[
	\frac{|A_4-A_9|}{11^{n-2}}
	\le
	11-\alpha_6
	\overset{\eqref{eq:beta5}}{<}
	\alpha_5+1\,,
\]
or, in other words, 
\begin{equation}\label{eq:2248}
	|A_4-A_9| < |A_5|+11^{n-2}\,.
\end{equation}
A symmetric argument also establishes 
\begin{equation}\label{eq:2249}
	|A_7-A_2| < |A_6|+11^{n-2}\,.
\end{equation}

By Inequality~\eqref{eq:224} we may assume that
\[
    \alpha_4+\alpha_5+\alpha_9>11.2>11+\tfrac{1}{11}\,.
\]
Together with Claim~\ref{nvmifempty} and Inequality~\eqref{eq:2248}
this shows that the sets $A_4$, $A_5$, and $A_9$ satisfy the 
assumptions of Lemma~\ref{lem:ABCD}. 
Consequently, we have 
\begin{equation}\label{eq:496}
	\alpha_4+\alpha_9\le 6
\end{equation} 
and $A_5-A_5=\FF_{11}^{n-1}$, whence $A_0=\vn$. 
For this reason,~\eqref{eq:1156} holds with equality, i.e.,
\[
    \beta_1+\beta_2+\beta_3+\beta_4+\beta_5=\alpha\ge 33\,.
\]
In view of Inequalities~\eqref{eq:1125} and~\eqref{eq:1150}
this shows 
\begin{align*}
		1.677\cdot 33
		&\le
     1.375\cdot\alpha+0.302\cdot(\beta_1+\beta_2+\beta_3+\beta_4)
     	+0.304\cdot\beta_5\\
    	&\le
		0.461\cdot(\beta_1+2\beta_5)
		+0.887\cdot(\beta_2+\beta_4+\beta_5)\\
		&\quad\qquad 
		+0.535\cdot(\beta_3+2\beta_4)+0.455\cdot(2\beta_3+\beta_5)\\
     &\le
     (0.461+0.535+0.455)\cdot 24
     +0.887\cdot(\beta_2+\beta_4+\beta_5)\,,
\end{align*}
whence 
\begin{equation}\label{eq:1424}
    \beta_2+\beta_4+\beta_5>23.1\,.
\end{equation}

By subtracting $\alpha_4+\alpha_5+\alpha_9\le 12$ we infer
\[
    \alpha_2+\alpha_6+\alpha_7>11.1>11+\tfrac 1{11}\,.
\]
Recalling~\eqref{eq:2249} we can apply Lemma~\ref{lem:ABCD}
to $A_7$, $A_6$, and $A_2$ as well, thereby deducing 
$\alpha_2+\alpha_7\le 6$.  
The addition of~\eqref{eq:496} gives $\beta_2+\beta_4\le 12$.
But now~\eqref{eq:1424} discloses $\beta_5>11$, which contradicts
Claim~\ref{nvmifempty}.
\end{proof}

\section{Fourier analysis}\label{sec:2}

Based on the results of the previous section we can now 
complete the proof of Theorem~\ref{thm:2} by means of a  
Fourier analytic argument. For definiteness we fix some 
notation. 

Given a finite abelian group $G$, we denote its {\it Pontryagin 
dual}, that is the group of homomorphisms from $G$ to the 
unit circle in the complex plane, by $\wh{G}$. For a function 
$f\colon G\lra\CC$, its {\it Fourier transform} 
$\wh{f}\colon\wh{G}\lra\CC$ is defined by 
\[
	\wh{f}(\gamma)=\sum_{x\in G}f(x)\overline{\gamma(x)}
\]
for all $\gamma\in\wh{G}$. 
Subsets of $G$ are identified with their characteristic 
functions. The following statement is well-known 
(cf.\ e.g.,~\cite{Gr05}*{Lemma~7.1}), but for the sake of 
completeness we provide a brief sketch of its proof. 

\begin{lemma}\label{lem:four}
	If $A$ denotes a sum-free subset of a finite abelian 
	group $G\ne 0$, then there exists a non-trivial 
	character $\gamma\in \wh{G}$ such that 
		\[
		\Real \wh{A}(\gamma)\le \frac{-|A|^2}{|G|-|A|}\,.
	\]
	\end{lemma} 

\begin{proof}
	Since the equation $a+a'-a''=0$ has no solutions 
	with $a, a', a''\in A$, we have 
		\[
		\sum_{\gamma\in \wh{G}}\bigl(\wh{A}(\gamma)\bigr)^2
		\wh{A}(-\gamma)=0\,,
	\]
		whence 
		\[
		\sum_{\gamma\in \wh{G}}\big|\wh{A}(\gamma)\big|^2\, 
		\Real \wh{A}(\gamma)=0\,.
	\]
		
	The trivial character contributes $|A|^3$ to the left side. 
	So if $K$ denotes the minimum value of $\Real \wh{A}(\gamma)$ 
	as $\gamma$ ranges over the non-trivial characters, we obtain 
		\[
		K\Bigl(-|A|^2+\sum_{\gamma\in \wh{G}}|\wh{A}(\gamma)|^2\Bigr)
		\le 
		-|A|^3\,.
	\]
		By Parseval's identity, the sum on the left side evaluates to $|G||A|$ 
	and the result follows.
\end{proof}

\begin{proof}[Proof of Theorem~\ref{thm:2}]
	Given a prime number $p=6m-1\ge 11$ and a dimension $n\ge 1$ 
	let $A\subseteq\FF_p^n$ be a sum-free set of 
	size $|A|\ge (2m-1)p^{n-1}$. 
	We need to show that $A$ is either structured or isomorphic 
	to a subset of $[2m-1, 4m]\times \FF_p^{n-1}$. 
	Assuming that this fails, Lemma~\ref{lem:5}
	informs us that $A\in\SFR_1(\FF_p^n)$ and Lemma~\ref{lem:6}
	yields~$n\ge 2$.

	By Lemma~\ref{lem:four} there exists 
	a non-trivial character $\gamma\in\wh{\FF_p^n}$ such that 
		\[
		\Real \wh{A}(\gamma)
		\le 
		\frac{-|A|^2}{p^n-|A|}\,.
	\]
		For reasons of symmetry we may assume that $\gamma$ sends each 
	point $(x_1, \dots, x_n)\in \FF_p^n$ to $e^{2x_1\pi i/p}$.
	In terms of the sets $A_0, \dots, A_{p-1}\subseteq\FF_p^{n-1}$ 
	determined by 
		\[	
		A=\bigdcup_{k\in\FF_p}\{k\}\times A_k
	\]
		we can rewrite the above estimate as 
		\begin{equation}\label{eq:2114}
		\Real \sum_{k=0}^{p-1}|A_k|e^{2k\pi i/p}
		\le 
		\frac{-|A|^2}{p^n-|A|}\,.
	\end{equation}
		Notice that by Proposition~\ref{prop:22} at most two 
	of the sets $A_k$ can be empty. 
	
	\smallskip

	{\it \hskip2em First Case: $p=11$}
        
	\smallskip
	
	By subtracting~\eqref{eq:2114} from $\sum_{k=0}^{10}|A_k|=|A|$ 
	and taking into account that $11^n-|A|\le 8\cdot 11^{n-1}$
	we infer 
		\[
		  \sum_{k=1}^{10}\left(1-\cos\frac{2k\pi}{11}\right)|A_k|
		  \ge
		  |A|+\frac{|A|^2}{11^n-|A|}
		  =
		  \frac{11^n|A|}{11^n-|A|}
		  \ge
		  \frac{11|A|}8\,.
	\]
		This contradicts Proposition~\ref{prop:24}. 
	
	\smallskip

	{\it \hskip2em Second Case: $p\ge 17$}
        
	\smallskip
	
	By Proposition~\ref{prop:23} there exists a real number $U$ 
	such that $\sum_{k\in \FF_p}\big||A_k|-U\big|\le 2p^{n-1}$. 
	Now we subtract $U\sum_{k=0}^{p-1}e^{2k\pi i/p}=0$ 
	from~\eqref{eq:2114} and take absolute values, thereby 
	obtaining
		\[
		\frac{(2m-1)^2p^{n-1}}{4m}
		\le 
		\frac{|A|^2}{p^n-|A|}
		\le 
		\bigg|\sum_{k=0}^{p-1}\bigl(|A_k|-U\bigr)e^{2k\pi i/p}\bigg|
		\le
		\sum_{k=0}^{p-1}\big||A_k|-U\big|
		\le
		2p^{n-1}\,,
	\]
		whence $4m(m-1)<(2m-1)^2\le 8m$. However, this gives $m<3$,
	which contradicts~$p\ge 17$.
\end{proof}    

\section{Concluding remarks}\label{sec:final}

\subsection{Small primes} 
As shown by Davydov and Tombak~\cite{DT} if for some $n\ge 4$ a 
sum-free set $A\subseteq \FF_2^n$ with $|A|>2^{n-1}+1$ is 
maximal with respect to inclusion, then it is periodic, i.e., 
there is a nonzero vector $v$ such that $A+v=A$. Furthermore, 
their results imply 
\[
	\sfr_k(\FF_2^n)=2^{n-2}+2^{n-3-k}
\]
whenever $k\ge 1$ and $n\ge k+3$. 
  
For $p=3$ Lev's periodicity conjecture from~\cite{VL05} asserts that 
if $A\subseteq \FF_3^n$ is a maximal sum-free set with 
$|A|>\frac12(3^{n-1}+1)$, then it is periodic. This has 
recently been proved in~\cite{R24c}. The classification of 
sum-free subsets $A\subseteq \FF_3^n$ with $|A|>\frac16\cdot 3^n$
provided there can be shown to entail 
\[
	\sfr_k(\FF_3^n)=\tfrac12(3^{n-1}+3^{n-k-1})
\]
whenever $k\ge 1$ and $n\ge k+2$.

Finally, for $p=5$ and $n\ge 3$ it is known 
that $\sfr_1(\FF_5^n)=28\cdot 5^{n-3}$. 
Moreover, there is a concrete set $\Lambda\subseteq \FF_5^3$
of size $28$ such that a set $A$ is in $\SFRR_1(\FF_5^n)$ 
if and only if it is of the form $\psi^{-1}[\Lambda]$ 
for some epimorphism $\psi\colon\FF_5^n\lra\FF_5^3$ 
(see~\cite{RZ24b}). Roughly speaking, this means that 
$\SFRR_1(\FF_5^n)$ is `very small' and consists of `highly 
structured' sets only. For this reason, we expect $\sfr_2(\FF_5^n)$
to be asymptotically smaller than $\sfr_1(\FF_5^n)$ and it 
would be very interesting to determine this number. 

\subsection{Primes of the form \texorpdfstring{$3m+1$}{{\it 3m+1}}}
Let $p=3m+1$ be a prime number. As shown by Rhemtulla and Street~\cite{RS70}, 
$\SFRR_0(\FF_p)$ consists of all sets isomorphic to $[m, 2m-1]$, 
$[m+1, 2m]$, or $[m, 2m+1]\sm\{m+1, 2m\}$. 
This yields $\sfr_1(\FF_7)=0$ and for $p\ge 13$ the set $[m-1, 2m-3]$
exemplifies $\sfr_1(\FF_p)=\sfr_0(\FF_p)-1=m-1$.

In~\cite{RS71} Rhemtulla and Street offer the following generalisation 
to higher dimensions. For $n\ge 2$ a set $A\subseteq \FF_p^n$
is in $\SFRR_0(\FF_p^n)$ if 
\begin{enumerate}
	\item[$\bullet$] either $A$ is isomorphic 
		to $[m+1, 2m]\times \FF_p^{n-1}$
\end{enumerate}
or there is a subspace $K$ of $\FF_p^{n-1}$ such that $A$ is 
isomorphic to one of 
\begin{enumerate}
	\item[$\bullet$] $\{m\}\times K\dcup [m+1, 2m-1]\times \FF_p^{n-1}
			\dcup \{2m\}\times(\FF_p^{n-1}\sm K)$
	\item[$\bullet$] or $\{m, 2m+1\}\times K
		\dcup \{m+1, 2m\}\times (\FF_p^{n-1}\sm K)
		\dcup [m+2, 2m-1]\times \FF_p^{n-1}$.
\end{enumerate}
Consequently, for every nonzero vector $x\in \FF_p^{n-1}$ 
the set 
\[
	\{(m, 0), (m, x)\}\dcup [m+1, 2m-1]\times \FF_p^{n-1}
	\dcup \{2m\}\times(\FF_p^{n-1}\sm\{0, x, 2x\})
\]
demonstrates $\sfr_1(\FF_p^n)=\sfr_0(\FF_p^n)-1=mp^{n-1}-1$.
The very small gap between $\sfr_1(\cdot)$ and $\sfr_0(\cdot)$ 
seems to indicate that determining $\sfr_k(\FF_p^n)$ for $k>1$
might not be very interesting when $p\equiv 1\pmod{3}$.

\subsection{More on primes of the form \texorpdfstring{$6m-1$}{{\it 6m-1}}}

Finally let us suppose again that $p=6m-1\ge 11$ is prime. 
For $p\ge 17$ the set $[2m-2, 4m-5]$ exemplifies 
$\sfr_2(\FF_p)=\sfr_1(\FF_p)-1=2m-2$, while a simple case analysis 
shows $\sfr_2(\FF_{11})=0$.
Similarly for $n\ge 2$ and every nonzero vector $x\in \FF_p^{n-1}$ 
the set 
\[
	\{(2m-1, 0), (2m-1, x)\}\dcup [2m, 4m-3]\times \FF_p^{n-1}
	\dcup \{4m-2\}\times (\FF_p^{n-1}\sm \{0, x, 2x\})
\]
witnesses $\sfr_2(\FF_p^n)=(2m-1)p^{n-1}-1=\sfr_1(\FF_p^n)-1$.
Accordingly for $k\ge 2$ we do not expect~$\sfr_k(\FF_p^n)$ 
to be very interesting. 

Summarising this discussion, the determination of $\sfr_2(\FF_5^n)$
seems to be the next natural problem in this area. It might also be
feasible to calculate $\sfr_1(G)$ for all finite abelian groups~$G$,
but we made no serious efforts in this direction. 

\subsection*{Acknowledgement} 
The second author would like to thank Leo Versteegen for introducing 
her to this topic and some early discussions.


\begin{bibdiv}
\begin{biblist}
\bib{Ba16}{article}{
   author={Balasubramanian, R.},
   author={Prakash, Gyan},
   author={Ramana, D. S.},
   title={Sum-free subsets of finite abelian groups of type III},
   journal={European J. Combin.},
   volume={58},
   date={2016},
   pages={181--202},
   issn={0195-6698},
   review={\MR{3530628}},
   doi={10.1016/j.ejc.2016.06.001},
}

\bib{Cau13}{article}{
	author={Cauchy, Augustin-Louis},
	title={Recherches sur les nombres},
	journal={J. \'Ecole Polytech.},
	volume={9},
	date={1813},
	pages={99--116},
}
	
\bib{Dav35}{article}{
   author={Davenport, H.},
   title={On the Addition of Residue Classes},
   journal={J. London Math. Soc.},
   volume={10},
   date={1935},
   number={1},
   pages={30--32},
   issn={0024-6107},
   review={\MR{1581695}},
   doi={10.1112/jlms/s1-10.37.30},
}

\bib{Dav47}{article}{
   author={Davenport, H.},
   title={A historical note},
   journal={J. London Math. Soc.},
   volume={22},
   date={1947},
   pages={100--101},
   issn={0024-6107},
   review={\MR{22865}},
   doi={10.1112/jlms/s1-22.2.100},
}

\bib{DT}{article}{
   author={Davydov\rn{(Davydov)}, A. A.},
   author={Tombak\rn{(Tombak)}, L. M.},
   title={\rn{Kvazisovershennye linei0nye dvoichnye kody s 
   rasstoj1niem 4 i polnye shapki v proektivnoi0 geometrii}},
   language={Russian},
   journal={\rn{Problemy peredachi informatsii}},
   volume={25},
   date={1989},
   number={4},
   pages={11--23},
   issn={0555-2923},
   translation={
      journal={Problems Inform. Transmission},
      volume={25},
      date={1989},
      number={4},
      pages={265--275 (1990)},
      issn={0032-9460},
   },
   review={\MR{1040020}},
}

\bib{EKR}{article}{
   author={Erd\H{o}s, P.},
   author={Ko, Chao},
   author={Rado, R.},
   title={Intersection theorems for systems of finite sets},
   journal={Quart. J. Math. Oxford Ser. (2)},
   volume={12},
   date={1961},
   pages={313--320},
   issn={0033-5606},
   review={\MR{140419}},
   doi={10.1093/qmath/12.1.313},
}

\bib{E65}{article}{
   author={Erd\H{o}s, P.},
   title={Extremal problems in number theory},
   conference={
      title={Proc. Sympos. Pure Math., Vol. VIII},
   },
   book={
      publisher={Amer. Math. Soc., Providence, RI},
   },
   date={1965},
   pages={181--189},
   review={\MR{174539}},
}

\bib{Gr05}{article}{
   author={Green, Ben},
   author={Ruzsa, Imre Z.},
   title={Sum-free sets in abelian groups},
   journal={Israel J. Math.},
   volume={147},
   date={2005},
   pages={157--188},
   issn={0021-2172},
   review={\MR{2166359}},
   doi={10.1007/BF02785363},
}

\bib{HK}{article}{
   author={Han, Jie},
   author={Kohayakawa, Yoshiharu},
   title={The maximum size of a non-trivial intersecting uniform 
   family that is not a subfamily of the Hilton-Milner family},
   journal={Proc. Amer. Math. Soc.},
   volume={145},
   date={2017},
   number={1},
   pages={73--87},
   issn={0002-9939},
   review={\MR{3565361}},
   doi={10.1090/proc/13221},
}

\bib{HM}{article}{
   author={Hilton, A. J. W.},
   author={Milner, E. C.},
   title={Some intersection theorems for systems of finite sets},
   journal={Quart. J. Math. Oxford Ser. (2)},
   volume={18},
   date={1967},
   pages={369--384},
   issn={0033-5606},
   review={\MR{219428}},
   doi={10.1093/qmath/18.1.369},
}

\bib{Kn53}{article}{
   author={Kneser, Martin},
   title={Absch\"{a}tzung der asymptotischen Dichte von Summenmengen},
   language={German},
   journal={Math. Z.},
   volume={58},
   date={1953},
   pages={459--484},
   issn={0025-5874},
   review={\MR{56632}},
   doi={10.1007/BF01174162},
}

\bib{Kn55}{article}{
   author={Kneser, Martin},
   title={Ein Satz \"{u}ber abelsche Gruppen mit Anwendungen auf die 
   	Geometrie der Zahlen},
   language={German},
   journal={Math. Z.},
   volume={61},
   date={1955},
   pages={429--434},
   issn={0025-5874},
   review={\MR{68536}},
   doi={10.1007/BF01181357},
}
			
\bib{VL05}{article}{
   author={Lev, Vsevolod F.},
   title={Large sum-free sets in ternary spaces},
   journal={J. Combin. Theory Ser. A},
   volume={111},
   date={2005},
   number={2},
   pages={337--346},
   issn={0097-3165},
   review={\MR{2156218}},
   doi={10.1016/j.jcta.2005.01.004},
}
	      
\bib{VL23}{article}{
   author={Lev, Vsevolod F.},
   title={Large sum-free sets in $\ZZ_5^n$},
   journal={J. Combin. Theory Ser. A},
   volume={205},
   date={2024},
   pages={Paper No. 105865, 9},
   issn={0097-3165},
   review={\MR{4700167}},
   doi={10.1016/j.jcta.2024.105865},
}

\bib{LLS}{article}{
   author={Lev, Vsevolod F.},
   author={\L uczak, Tomasz},
   author={Schoen, Tomasz},
   title={Sum-free sets in abelian groups},
   journal={Israel J. Math.},
   volume={125},
   date={2001},
   pages={347--367},
   issn={0021-2172},
   review={\MR{1853817}},
   doi={10.1007/BF02773386},
}

\bib{Joanna}{article}{
   author={Polcyn, Joanna},
   author={Ruci\'{n}ski, Andrzej},
   title={A hierarchy of maximal intersecting triple systems},
   journal={Opuscula Math.},
   volume={37},
   date={2017},
   number={4},
   pages={597--608},
   issn={1232-9274},
   review={\MR{3647803}},
   doi={10.7494/OpMath.2017.37.4.597},
}

\bib{RS70}{article}{
   author={Rhemtulla, A. H.},
   author={Street, Anne Penfold},
   title={Maximal sum-free sets in finite abelian groups},
   journal={Bull. Austral. Math. Soc.},
   volume={2},
   date={1970},
   pages={289--297},
   issn={0004-9727},
   review={\MR{263920}},
   doi={10.1017/S000497270004199X},
}

\bib{RS71}{article}{
   author={Rhemtulla, A. H.},
   author={Street, Anne Penfold},
   title={Maximal sum-free sets in elementary abelian $p$-groups},
   journal={Canad. Math. Bull.},
   volume={14},
   date={1971},
   pages={73--80},
   issn={0008-4395},
   review={\MR{292936}},
   doi={10.4153/CMB-1971-014-2},
}
	
\bib{TV17}{article}{
   author={Tao, Terence},
   author={Vu, Van},
   title={Sum-free sets in groups: a survey},
   journal={J. Comb.},
   volume={8},
   date={2017},
   number={3},
   pages={541--552},
   issn={2156-3527},
   review={\MR{3668880}},
   doi={10.4310/JOC.2017.v8.n3.a7},
}

\bib{R24c}{article}{
	author={Reiher, Chr.},
	title={On Lev's periodicity conjecture},
	note={Unpublished manuscript},
}

\bib{RZ24b}{article}{
	author={Reiher, Chr.},
	author={Zotova, Sofia},
	title={Large sum-free sets in finite vector spaces II.},
	note={Unpublished manuscript},
}

\bib{Schur}{article}{
	author={Schur, Issai},
	title={\"{U}ber die Kongruenz $x^m+y^m\equiv z^m \pmod{p}$},
	journal={Deutsche Math. Ver.},
	volume={25},
	date={1916},
	pages={114--117},
} 

\bib{LV23}{article}{
	author={Versteegen, Leo},
	title={The structure of large sum-free sets in $\FF_p^n$},
	eprint={2303.00828},
}
   
\bib{Vos56a}{article}{
   author={Vosper, A. G.},
   title={The critical pairs of subsets of a group of prime order},
   journal={J. London Math. Soc.},
   volume={31},
   date={1956},
   pages={200--205},
   issn={0024-6107},
   review={\MR{77555}},
   doi={10.1112/jlms/s1-31.2.200},
}
	
\bib{Vos56b}{article}{
   author={Vosper, A. G.},
   title={Addendum to ``The critical pairs of subsets of a group of 
   	prime order''},
   journal={J. London Math. Soc.},
   volume={31},
   date={1956},
   pages={280--282},
   issn={0024-6107},
   review={\MR{78368}},
   doi={10.1112/jlms/s1-31.3.280},
}

\bib{Yap}{article}{
   author={Yap, H. P.},
   title={Maximal sum-free sets of group elements},
   journal={J. London Math. Soc.},
   volume={44},
   date={1969},
   pages={131--136},
   issn={0024-6107},
   review={\MR{232844}},
   doi={10.1112/jlms/s1-44.1.131},
}

\end{biblist}
\end{bibdiv}
\end{document}